\documentclass{amsart}
\usepackage{graphicx}
\usepackage{amscd}
\usepackage{amsmath}
\usepackage{amsfonts}
\usepackage{amssymb}
\theoremstyle{plain}
\newtheorem{theorem}{Theorem}
\newtheorem{corollary}[theorem]{Corollary}
\newtheorem{lemma}[theorem]{Lemma}
\newtheorem{proposition}[theorem]{Proposition}
\theoremstyle{definition}
\newtheorem{example}[theorem]{Example}

\newtheorem{definition}[theorem]{Definition}

\newtheorem{remark}[theorem]{Remark}
\theoremstyle{remark}

\begin{document}
\title{Perinormal rings with zero divisors}

\author{Tiberiu Dumitrescu and Anam Rani}

\address{Facultatea de Matematica si Informatica,University of Bucharest,14 A\-ca\-de\-mi\-ei Str., Bucharest, RO 010014, Romania}
\email{tiberiu@fmi.unibuc.ro, tiberiu\_dumitrescu2003@yahoo.com}

\address{Abdus Salam School of Mathematical Sciences, GC University,
Lahore. 68-B, New Muslim Town, Lahore 54600, Pakistan} \email{anamrane@gmail.com}

\begin{abstract}
We  extend to rings with zero-divisors the concept of perinormal domain introduced by N. Epstein and J. Shapiro. A ring  $A$ is  called perinormal if every  overring of $A$ which satisfies going down over $A$ is $A$-flat.
The Pr\"ufer rings and the Marot Krull rings are perinormal.
\end{abstract}


.

\maketitle

\section{Introduction}

In their  paper \cite{ES}, N. Epstein and J. Shapiro introduced and studied   {perinormal domains}, that is, those domains whose going down overrings are flat $A$-modules, cf. \cite[Proposition 2.4]{ES}. Here  a going down overring $B$ of $A$ means a ring between $A$ and its quotient field such that going down holds for $A\subseteq B$ (see \cite[page 31]{M}).
Among other results, they showed that perinormality is a local property  \cite[Theorem 2.3]{ES}, a localization of a perinormal domain at a height one prime ideal is a valuation domain  \cite[Proposition 3.2]{ES} and every (generalized) Krull domain is perinormal \cite[Theorem 3.10]{ES}. They also  characterized the universally catenary Noetherian perinormal domains \cite[Theorem 4.7]{ES}.

In \cite{DR}, we extended some of the results in \cite{ES} using the concept of a P-domain introduced by Mott and Zafrullah in \cite{MZ}. A domain $A$ is called a {\em P-domain} if $A_P$ is a valuation domain for every  prime ideal  $P$ which is minimal over an ideal of the form $Aa:b$ with $a,b\in A-\{0\}$.
The Pr\"ufer v-multiplication domains (hence the Krull domains and the GCD domains) are P-domains \cite[Theorem 3.1]{MZ}.
Recall that  $A$ is a {\em Pr\"ufer v-multiplication domain} if for every  finitely generated nonzero ideal $I$, there exists a finitely generated nonzero ideal $J$ such that the divisorial closure of $IJ$ is a principal (see \cite[Section 34]{G} and \cite{MZ}). The main result of \cite{DR} asserts that every P-domain is perinormal \cite[Theorem]{DR}.

The purpose of this paper is to study {\em perinormal rings with zero-divisors}. Almost all properties of perinormal domains can be recovered for perinormal rings in some appropriate setup.
Let $A$ be a ring, $Reg(A)$ the set regular elements (that is, non-zero-divisors) of $A$ and $K$ the total quotient ring of $A$, that is, the fraction ring of $A$ with denominators in $Reg(A)$. An ideal of $A$ is called {\em regular} if it contains regular elements. As usual, by  an {\em overring} $B$ of $A$ we mean a ring between $A$ and $K$. We say that $B$ is a {\em GD-overring} if $A\subseteq B$ satisfies going down. We call a ring  $A$  {\em perinormal} if every  GD-overring of $A$  is flat.

In Section $2$ we prove some basic facts about the perinormal rings.
Every Pr\"ufer ring is perinormal (Proposition \ref{1}).
While perinormality is not a local property in the classical sense (Example \ref{26}), it is a regular local property (Proposition \ref{4}). More precisely,  a ring $A$ is perinormal if and only if $A_{(M)}$ is perinormal for each maximal ideal $M$ of $A$. Here $A_{(M)}$ is the regular quotient ring \cite[page 28]{H}, that is, the fraction ring of $A$ with denominators in  $S=(A-M)\cap Reg(A)$.
A perinormal ring has no proper integral unibranched overrings (Proposition \ref{5}). Here a ring extension $A\subseteq B$ is called unibranched if the contraction map  $Spec (B)\rightarrow Spec (A)$ is  bijective.
A direct product ring $A\times B$ is a perinormal ring if and only if $A$ and $B$ are perinormal (Theorem \ref{8}). In particular, the integral closure  of a reduced Noetherian ring  is  perinormal (Corollary \ref{27}).

For the rest of the paper, we assume that   all rings are Marot rings, that is, their regular ideals are generated by regular elements, cf. \cite[page 31]{H}. In Section $3$ we extend the main result of \cite{DR} for rings with zero-divisors.
We call a ring $A$  a  {\em $P$-ring} if $A_{(P)}$ is a (Manis) valuation ring on $K$ whenever $P$ is minimal over  $bA:a$ for some $a\in A$ and $b\in Reg(A)$.
The Krull rings are P-rings. Adapting some material from \cite{MZ}  for rings with zero-divisors, we prove that every  $P$-ring  is perinormal (Theorem \ref{23}).

The main result of Section $4$, Theorem \ref{28}, extends for rings with zero-divisors the  characterization given in \cite[Theorem 4.7]{ES} for the universally catenary Noetherian perinormal domains.
To this aim, we  extend  a result of McAdam \cite[Theorem 2]{Mc} for rings with zero-divisors: if $A$ is a Noetherian ring and  $B$ an integral
GD-overring of $A$, then  the non-minimal regular prime ideals of $A$  are unibranched in $B$ (Theorem \ref{17}).

Throughout this paper all rings are commutative unitary rings. Any unexplained terminology is standard like in \cite{AM}, \cite{G}, \cite{H} or \cite{LM}. For basic facts on rings with zero-divisors we refere the reader to \cite{H}.

\section{Basic facts}

We write again the key definition of this paper.

\begin{definition}A ring  $A$ is called {\em perinormal} if every  GD-overring of $A$  is flat (as an $A$-module).\end{definition}

Note that a domain $D$ is a perinormal ring if and only if $D$ is a perinormal domain in the sense of  \cite{ES}. A total quotient ring (e.g. a zero-dimensional ring) is a trivial example of perinormal ring.
Recall  that a ring $A$ is a {\em Pr\"ufer} ring if its  finitely generated regular ideals are invertible, cf. \cite[Chapter II]{H}. A regular ideal $I$ of $A$ is {\em invertible} if $II^{-1}=A$ where $I^{-1}=\{ x\in K\mid xI\subseteq A\}$ and $K$ is the total quotient ring of $A$, cf. \cite[page 29]{H}. By \cite[Theorem 13]{Gf}, a ring $A$ is Pr\"ufer if and only if its overrings  are $A$-flat, so we have:

\begin{proposition}\label{1}
Every Pr\"ufer ring is perinormal.
\end{proposition}

Let $A$ be a ring with total quotient ring $K$ and $P$ a prime ideal of $A$. Recall  that besides $A_P$, which in general is not an overring of $A$, it is customarily to consider the following two overrings of $A$:
the {\em large quotient ring} $A_{[P]}=\{ x\in K\mid  sx\in A$ for some $s\in A-P\}$ and the  {\em regular quotient ring} $A_{(P)}$ which is the fraction ring $A_S$ where $S=(A-P)\cap Reg(A)$, cf. \cite[page 28]{H}. We have $A\subseteq A_{(P)} \subseteq A_{[P]}\subseteq K$. As shown in \cite[Theorem 4.5]{Bo}, $A_{[P]}$ is not always $A$-flat.
In the next lemma, we recall some well-known facts about flatness.

\begin{lemma}\label{2}
Let $A$ be a ring and  $B\subseteq C$ two overrings of $A$.

$(a)$ The extension $A\subseteq B$  is flat if and only if $A_{[M\cap A]}=B_{[M]}$  for all maximal ideals $M$ of $B$.

$(b)$  If $A\subseteq B$ and $B\subseteq C$ are flat, then $A\subseteq B$ is flat.

$(c)$  If $A\subseteq C$ is flat, then $B\subseteq C$ is flat.

$(d)$ If  $A\subseteq B$  is flat and integral, then $A=B$.
\end{lemma}
\begin{proof}
$(a)$  is a part of \cite[Proposition 10]{Gf}.
$(b)$ The transitivity of flatness holds in general (see for instance \cite[Theorem (3.B)]{M}).
$(c)$ Let $M\in Max (C)$ and set $Q:=M\cap B$, $P:=Q\cap A$.
By $(a)$ we have $C_{[M]}=A_{[P]}\subseteq B_{[Q]}\subseteq C_{[M]}$, so
$B_{[Q]}=C_{[M]}$, thus  $B\subseteq C$ is flat applying $(a)$ again.
$(d)$ is \cite[Proposition 12]{Gf}.
\end{proof}


\begin{proposition}\label{3}
Let $A$  be a perinormal ring. If   $B$ is a flat overring of $A$, then $B$ is perinormal. In particular, $A_{(P)}$ is perinormal for each  prime ideal  $P$ of $A$.
\end{proposition}
 \begin{proof}
Let $C$ be a GD-overring of $B$. Since $B$ is $A$-flat it is a GD-overring of $A$, so $C$ is a GD-overring of $A$. As $A$ is perinormal, then $A\subseteq C$ is flat, so $B\subseteq C$ is flat by Lemma \ref{2}. Thus $B$ is perinormal.
\end{proof}

\begin{example}\label{26}
There exist a perinormal ring $A$ having a prime ideal $P$ such that $A_P$ is not perinormal. Indeed, let $(B,M)$ be a local one-dimenional domain which is not integrally closed (e.g. $\mathbb{Q}[[t^2,t^3]]$), $C:=B[X,Y]/X(M,X,Y)$ with $X,Y$ are indeterminates and $A:=C_N$ where $N=(M,X,Y)$.
Then $A$ is local with the maximal ideal $NA_N$ annihilated by $X$, so $A$ equals its total quotient ring, hence $A$ is perinormal. On the other hand, if $P$ is the prime ideal $(M,X)A$, then $A_P\simeq B(Y)$  is a one-dimenional domain which is not integrally closed (see \cite[Section 33]{G}), hence $A_P$ is not perinormal by \cite[Proposition 3.2]{ES}.
\end{example}

\begin{proposition}\label{4}
A ring $A$ is perinormal if and only if $A_{(M)}$ is perinormal for each maximal ideal $M$ of $A$.
\end{proposition}
\begin{proof}
The implication $(\Rightarrow)$ is covered by Proposition \ref{3}.  $(\Leftarrow)$ Let $B$ be a GD-overring of $A$, $M\in Max A$ and set $S=(A-M)\cap Reg(A)$. Then $B_{S}$ is a GD-overring of $A_S=A_{(M)}$. Since $A_{(M)}$ is perinormal, it follows that $A_S\subseteq B_S$ is flat, hence $A_M\subseteq B_M=B\otimes_A A_M$ is flat, because  $M$ is disjoint of $S$, so $(A_S)_{MA_S}\simeq A_M$.
Thus $A\subseteq B$ is flat by \cite[Theorem (3.J)]{M}.
\end{proof}

\begin{remark}
The preceding proof can be easily adapted to show that $A$ is perinormal whenever $A_M$ is perinormal for each maximal ideal $M$.
\end{remark}


Recall that a ring extension $A\subseteq B$ is {\em unibranched} if the contraction map  $Spec (B)\rightarrow Spec (A)$ is  bijective.
The following result extends \cite[Proposition 3.3]{ES} and a part of \cite[Corollary 3.4]{ES}.

\begin{proposition}\label{5}
Let $A$ be a perinormal ring with total quotient ring $K$. If $B$ is an integral unibranched overring of $A$, then $A=B$. In particular, if $b\in K$ and $b^2,b^3\in A$, then $b\in A$.
\end{proposition}
\begin{proof}
The contraction map $F:Spec(B)\rightarrow Spec(A)$ is not only a bijection but also an isomorphism of ordered sets. Indeed, let $Q_1,Q_2\in Spec(B)$ such that $F(Q_1)\subseteq F(Q_2)$. By going up, there exists  $Q_3\in Spec(B)$ with $Q_1\subseteq Q_3$ such that $F(Q_2)=F(Q_3)$, so $Q_1\subseteq Q_2=Q_3$ because $F$ is injective. Now is clear that $A\subseteq B$ satisfies going down. Since $A$ is perinormal,  we get that $A\subseteq B$ is  flat, so $A=B$, cf. \cite[Proposition 12]{Gf}. The 'in particular' assertion now easily
follows because  $A\subseteq A[b]$ is unibranched, cf. \cite[Lemma 2.4]{S}.
\end{proof}

A reduced ring $A$ satisfying the conclusion of the 'in particular' assertion above is called {\em seminormal}, see for instance the paragraph after Corollary 3.4 in \cite{S}.
\\[2mm] In order to prove that a direct product ring $A\times B$ is a perinormal ring if and only if $A$ and $B$ are perinormal, we  need the following well-known result and a simple lemma.

\begin{proposition}  \label{14}
The ring extensions $A\subseteq C$ and  $B\subseteq D$ are flat if and only if
$A\times B \subseteq C\times D$ is flat.
\end{proposition}

\begin{lemma}\label{7}
Let $A\subseteq C$ be a ring extension and $B$ a ring. Then  $A\subseteq C$ satisfies going down if and only if   $A \times B \subseteq C \times B$ satisfies going down.
\end{lemma}
\begin{proof}
$(\Rightarrow)$ Let $P\subseteq Q$ with  $P \in Spec (A\times B)$ and $Q \in Spec (C\times B)$. We examine the two cases: $(1)$ $Q=Q_1\times B$ with $Q_1\in Spec(C)$ and $(2)$ $Q=C\times M_1$ with $M_1\in Spec(B).$
In case 1, we get $P=P_1\times B$ with $P_1\in Spec(A)$, so $P_1\subseteq Q_1$. As going down holds for $A\subseteq C$, there exists a prime ideal $Q_2\subseteq Q_1$ of $C$ with $Q_2\cap A=P_1$. So $Q_2\times B\subseteq Q$ lies over $P$. In case 2, we get $P=A\times M_2$ with $M_2\in Spec(B)$, so $C\times M_2\subseteq Q$ lies over $P$.

$(\Leftarrow)$ Let $Q \in Spec (C)$ and $P \in Spec (A)$ with $P\subseteq Q$.
As going down holds for $A\times B\subseteq C\times B$, there exists a prime ideal $Q_1\times B\subseteq Q\times B$ of $C\times B$ lying over $P\times B$, so $Q_1\subseteq Q$ is a prime of $C$ lying over $P$.
\end{proof}

\begin{theorem}\label{8}
Let $A$ and $B$ be  rings. Then  $A\times B$ is perinormal if and only $A$ and $B$ are perinormal.
\end{theorem}
\begin{proof}
$(\Rightarrow)$ We show that $A$ is perinormal. Let $C$ be a GD-overring of $A$. By Lemma \ref{7}, $A \times B \subseteq C \times B$ satisfies going down. Since $A\times B$ is perinormal, $A \times B\subseteq C \times B$ is flat, so $A\subseteq C$ is flat by Lemma \ref{14}. Hence $A$ is perinormal.

$(\Leftarrow)$  Let $M$ be a prime ideal of $A\times B$.  By symmetry, we may assume that  $M=P\times B$ with $P\in Spec(A)$. By Proposition \ref{4}, it suffices to show that $(A\times B)_{(P\times B)}=A_{(P)}\times L$ is perinormal, where $L$ is the total quotient ring of $A$. Let $C$ be a GD-overring of $A_{(P)}\times L$. It is easy to see that $C=D\times L$ where $D$ is an overring of $A_{(P)}$. By Lemma \ref{7}, $D$ is a GD-overring of $A_{(P)}$. As $A_{(P)}$ is perinormal, we get that $A_{(P)}\subseteq D$ and hence $A_{(P)}\times L\subseteq C$ are both flat, cf. Lemma \ref{14}.
\end{proof}

Recall  that a domain $A$ is a  {\em Krull domain}  if $A=\cap_{P\in X^{1}(A)} A_P$, this intersection has finite character and $A_P$ is  a discrete valuation domain for each $P\in X^{1}(A)$ (see for instance \cite[Section 43]{G}). By \cite[Theorem 3.10]{ES}, a Krull domain is perinormal.

\begin{corollary}\label{27}
The integral closure $A^\prime$ of a reduced Noetherian ring $A$ is  perinormal.
\end{corollary}
\begin{proof}
By \cite[Theorem 10.1]{H}, $A^\prime$ is finite direct product of  Krull domains. Apply Theorem \ref{8} and \cite[Theorem 3.10]{ES}.
\end{proof}

\noindent The preceding result holds also for non-reduced Noetherian rings, cf. Corollary \ref{20}.

\section{P-rings and Krull rings}


By \cite[Theorem 3.10]{ES}, a Krull domain is perinormal. This result was extended in \cite[Theorem 2]{DR} using the concept of a P-domain.
According  to \cite[page 2]{MZ}, a domain $A$ is called a {\em P-domain} if $A_P$ is a valuation domain for every  prime ideal  $P$ which is minimal over an ideal of the form $Aa:b$ with $a,b\in A-\{0\}$.
In  \cite[Theorem 2]{DR},  we proved that every P-domain  is perinormal. In particular a GCD domain or a Krull domain is perinormal. The aim of this section is to extend \cite[Theorem 2]{DR} to rings with zero-divisors.

Recall  that a {\em Marot ring} is a ring whose regular ideals are generated by regular elements, cf. \cite[page 31]{H}. Each overring of a Marot ring is a Marot ring (cf. \cite[Corollary 7.3]{H}),  every Noetherian ring is Marot (cf. \cite[Theorem 7.2]{H}) and, if $A$ is Marot, then $A_{(P)} = A_{[P]}$ for each prime ideal $P$ of $A$ (cf. \cite[Theorem 7.6]{H}).

{\em For the rest of this paper, we assume that  all rings are Marot rings}.
\\[1mm]
Let $A$ be a ring with total quotient ring $K$ and $M$ an $A$-module.
Recall  that a prime ideal $P$ of  $A$ is called {\em weakly  Bourbaki associated prime} to $M$ \cite[page 289]{B}, if $P$ is minimal over the annihilator ideal $Ann(x)$ for some $x\in M$. Denote by $Ass(M)$ the set of weakly Bourbaki associated primes of $M$. In particular, for the factor $A$-module $K/A$ we have
$Ass(K/A)=\{P\in Spec(A)\mid$  there exist $a\in A$, $b\in Reg(A)$ such that $P$ is minimal over  $bA:a\}$. Note that each $P\in Ass(K/A)$ is regular because $bA:a$ contains the regular element $b$.
The following fact is well-known for domains.

\begin{proposition}\label{18}
Let $A$ be a  ring with total quotient ring $K$. Then

$$A=\bigcap\limits _{P\in Ass(K/A)}A_{(P)}.$$
\end{proposition}
\begin{proof}
By contrary, assume there exist $a\in A$ and $b\in Reg(A)$ such that $a/b\in \bigcap\limits _{P\in Ass(K/A)}A_{(P)}-A$. Then $bA:a$ is a proper ideal and, if we pick  a minimal prime ideal $Q$ of $bA:a$, then $Q\in Ass(K/A)$.
By Remark \ref{19}, $a/b\notin A_{(Q)}$ which is a contradiction. The other inclusion is obvious.
\end{proof}

\begin{remark}\label{19}
Let $A$ be a ring, $a\in A$, $b\in Reg(A)$ and $P\in Spec(A)$. Clearly, if $a/b\in A_{(P)}$, then  $bA:a\not\subseteq P$. Conversely, if $bA:a\not\subseteq P$, there exists some regular element in $(bA:a)- P$, so
$a/b\in A_{(P)}$, because $A$ is a Marot ring.
\end{remark}


We briefly recall the definition of a (Manis) valuation ring. For details we refer the reader to \cite[Chapter II]{H}.  Let $K$ be a ring. A {\em valuation} on $K$ is a surjective map $v:K\rightarrow G\cup \{\infty\}$,
where $(G,+)$ is a totally ordered abelian group, such that

$(a)$ $v(xy)=v(x)+v(y)$ for all $x,y\in K$,

$(b)$ $v(x+y)\geq min \{v(x),v(y)\}$ for all $x,y\in K$,

$(c)$ $v(1)=0$ and $v(0)=\infty$.

\begin{theorem}{\em (\cite[Theorem 5.1]{H})} \label{10}
Let $A$ be a ring with total quotient ring $K$ and let $M$ be a prime ideal of $A$. Then the following conditions are equivalent.

$(a)$ If $B$ is an overring of $A$ having  a prime ideal $N$ such that $N\cap A=M$, then $A=B$.

$(b)$ If $x\in K-A$, there exists $y\in M$ such that $xy\in A-M$.

$(c)$ There exists a valuation $v$ on $K$ such that
$A=\{ x\in K\mid v(x)\geq(0)\}$ and $M=\{ x\in K\mid v(x)>0\}$.
\end{theorem}

A pair $(A,M)$ satisfying the equivalent conditions of Theorem \ref{10} is called a {\em valuation pair} and $A$ is called a {\em valuation ring} (on $K$). By \cite[Corollary 10.5]{LM}, $M$ is uniquely determined by $A$ if $A\neq K$.
If $G$ is the group of integers, then $A$ is called a {\em discrete valuation ring (DVR)}.

\begin{definition}
$(a)$ Let $A$ be a  ring and $P$ a regular prime ideal of $A$. We  say that $P$ is {\em valued} if $(A_{(P)},PA_{(P)})$ is a valuation pair.

$(b)$ A ring $A$ is called an {\em essential ring} if $A=\bigcap_{P\in G}A_{(P)}$ where $G$ is a set of  valued primes of $A$.

$(c)$ A ring $A$ is called a {\em $P$-ring} if every $Q\in Ass_A(K/A)$ is a valued prime.
\end{definition}

The following result extends \cite[Proposition 1.1]{MZ} for rings with zero-divisors.

\begin{theorem}\label{12}
Let $A$ be a  ring.

$(a)$ If $A$ is a P-ring, then $A$ is essential.

$(b)$ If $A$ is a P-ring and $S\subseteq Reg(A)$ a multiplicative set, then
 $A_S$ is a P-ring.

$(c)$ If $A_{(M)}$ is an essential ring for each $M\in Ass(K/A)$, then $A$ is a P-ring.
\end{theorem}
\begin{proof}
$(a)$ follows from Proposition \ref{18}.
$(b)$ Set $B:=A_{S}$ and let $Q\in Ass_{B}(K/B)$. By \cite [page 289]{B},  $Q=PB$ for some $P\in Ass_{A}(K/A)$ disjoint of $S$. Hence $B_{(Q)}=B_{(PB)}=A_{(P)}$ is a valuation ring.

$(c)$ Suppose that  $M\in Ass_{A}(K/A)$ and set $B=A_{(M)}$. Then $M$ is a  minimal prime over $bA:a$ for some $a\in A$ and  $b\in Reg(A)$, so $a/b\not\in  B$ by Remark \ref{19}.  Since $B$ is essential, there exists a valued prime $N$ of $B$ such that $a/b\notin B_{(N)}$. As $A$ is a Marot ring, $N=QB$ for some $Q\in Spec(A)$ contained in $M$.   Since $B_{(N)}=A_{(Q)}$ is a valuation ring, it follows that $Q$ is a valued prime. From $a/b\notin B_{(N)}$ and Remark \ref{19}, we get $bA:a\subseteq Q\subseteq M$, hence $M=Q$ is a valued prime, because $M$ is minimal  over $bA:a$.
\end{proof}

\begin{remark}\label{25}
Let $A$ be a ring with total quotient ring $K$ and let $B$ be an overring of $A$. It is well-known that the contraction map $Spec(B)\rightarrow Spec(A)$ induces a one-to-one correspondence between the non-regular prime ideals of $B$ and the non-regular prime ideals of $A$ whose inverse is given by $P\mapsto PK\cap B$ (see \cite[page 2]{H}). Consequently, the contraction map $Spec(B)\rightarrow Spec(A)$ is surjective if and only if every regular prime ideal of $A$ is lain over by some (necessarily regular) prime ideal of $A$.
\end{remark}

The following result is probably well-known.

\begin{proposition}\label{9}
Let $A$ be a ring, $B$ a GD-overring of $A$, $Q\in Spec (B)$ and $P:=Q\cap A$. Then the contraction map ${Spec} (B_{(Q)})\rightarrow Spec (A_{(P)})$ is surjective.
\end{proposition}
\begin{proof}
We may assume that $Q$ is regular, otherwise $A_{(P)}=B_{(Q)}$ is the total quotient ring of $A$.
By Remark \ref{25}, it suffices consider a regular prime ideal $N$ of $A_{(P)}$.  Note that $M:=N\cap A$ is regular.
Since $M$ is disjoint of $(A-P)\cap Reg(A)$, we get that every regular element of $M$ is in $P$, so $M\subseteq P$ because $A$ is a Marot ring. By going down, there is some $Q_1\in Spec (B)$ such that $Q_1\cap A = M$ and $Q_1\subseteq Q$. As $M$ is regular so is  $Q_1$. Thus $Q_1B_{(Q)}$ is a regular prime ideal of $B_{(Q)}$  that  lies over $N$.
\end{proof}

The next result extends \cite[Theorem 1]{DR} for rings with zero-divisors.

\begin{proposition}\label{15}
Let $A$ be an essential ring and $B$ an overring of $A$. If the contraction map  $Spec (B)\rightarrow Spec (A)$ is surjective, then $A=B$.
\end{proposition}
\begin{proof}
Let $G$ be the set of valued primes of $A$. Let $P\in G$ and pick $Q\in Spec(B)$ such that $Q\cap A=P$. Since $A_{(P)}\subseteq B_{(Q)}$,  $(A_{(P)}, PA_{(P)})$ is a valuation pair and $QB_{(Q)}$ lies over $PA_{(P)}$, we obtain $B\subseteq B_{(Q)}=A_{(P)}$, cf. Theorem \ref{10}. Since $A$ is essential, we get $B\subseteq \cap_{P\in G}A_P=A$, so $A=B$.
\end{proof}

The next result extends \cite[Theorem 2]{DR} for rings with zero-divisors.

\begin{theorem}\label{23}
Every  $P$-ring $A$ is perinormal.
\end{theorem}
\begin{proof}
Denote by $K$ the total quotient ring of $A$.
Let  $B$ a GD-overring of $A$, $Q$ a maximal ideal of $B$ and $P:=Q\cap A$.
If $Q$ is non-regular, then so is $P$ and we get easily that $A_{(P)}=K=B_{(Q)}$. Assume that $Q$ and hence $P$ are regular ideals.
 By Theorem \ref{12} and Proposition \ref{9}, $A_{(P)}$ is an essential ring and the contraction map ${Spec} (B_{(Q)})\rightarrow Spec (A_{(P)})$ is surjective.
By Proposition \ref{15}, we get $A_{(P)}=B_{(Q)}$.
Thus $A\subseteq B$ is flat due to part $(a)$ of Lemma \ref{2}.
\end{proof}

Say that a  prime ideal of a ring is {\em minimal regular} if it is minimal in the set of regular primes (see \cite[page 40]{H}).
Let $A$ be a ring and ${H}$ be the set of minimal regular prime ideals of $A$. Recall from \cite[Section 7]{Md} that   $A$ is a  {\em Krull ring}  if

$(1)$ $A=\cap_{P\in {H}} A_{(P)}$,

$(2)$ each  $x\in Reg(A)$ belongs to finitely many prime ideals $P\in H$,

$(3)$ $A_{(P)}$ is a DVR  for each $P\in {H}$.
\\[2mm]
Clearly a Krull ring is essential. By \cite[Theorem 10.1]{H}, the integral closure of a Noetherian ring is a Krull ring.
By \cite[Theorem 8.2]{H}, $A_{(P)}$ is a Krull ring whenever $A$ is a Krull ring and $P$ a prime ideal of $A$. So Theorem \ref{12} applies to show that a Krull ring is a P-ring. Thus we have the following corollary of
Theorem \ref{23} which extends the Krull part of \cite[Theorem 3.10]{ES}.

\begin{corollary}\label{20}
Every Krull ring is perinormal. In particular, the  integral closure of a Noetherian ring is perinormal.
\end{corollary}

\section{Noetherian perinormal rings}

According to \cite{ES}, a domain $D$ {\em satisfies $(R_1)$} if $D_P$ is
a valuation domain whenever $P$ is a height one prime of $D$. By \cite[Proposition 3.2]{ES},  any perinormal domain  satisfies $(R_1)$.
One of the main results of \cite{ES}, Theorem 4.7, gives the following  description of the perinormal universally catenary Noetherian domains.

\begin{theorem} {\em (\cite[Theorem 4.7]{ES})} \label{24}
Let $D$ be a universally catenary Noetherian domain. The following are equivalent.

$(a)$ $D$ is perinormal.

$(b)$ $D$ satisfies $(R_1)$ and, for each $P\in  Spec (D)$, $D_P$ is the only ring
$C$ between $D_P$ and its integral closure such that $D_P\subseteq C$ is unibranched.
\end{theorem}

The universal catenary is used to derive that: $(\sharp)$ every height one prime ideal of the integral closure of $D$ contracts to a height one prime ideal of $D$, cf. \cite[Lemma 4.3]{ES}.  Condition $(\sharp)$ is clearly stable under localizations, because $(D_P)'=D'_P$.
Since every perinormal domain satisfies $(R_1)$, we could merge conditions  $(R_1)$ and $(\sharp)$ in order to rephrase Theorem \ref{20} as follows:

\begin{theorem} \label{21}
Let $D$ be a  Noetherian domain  with integral closure $D'$.
Assume that $D_{P\cap A}$ is a DVR for every height one prime ideal $P$ of $D'$. The following are equivalent.

$(a)$ $D$ is perinormal.

$(b)$ For each $P\in  Spec (D)$, $D_P$ is the only ring $C$ between $D_P$ and its integral closure such that $D_P\subseteq C$ is unibranched.
\end{theorem}

The purpose of this section is to extend this result for rings with zero-divisors. Our approach is not to adapt the proof of  \cite[Theorem 4.7]{ES} but to adapt and use the following result of McAdam \cite[Theorem 2]{Mc}. Recall that if $A\subseteq B$ is a ring extension, a  prime ideal $P$ of $A$ is called  {\em unibranched} in $B$ if there exists a unique prime ideal of $B$ lying over $P$.

\begin{theorem} {\em (\cite[Theorem 2]{Mc})}\label{22}
Let $A$ be a Noetherian domain and $B$   an integral overring of $A$. Then $A\subseteq B$ has going down if and only if every prime ideal of $A$ of height $\geq 2$ is unibranched in $B$.
\end{theorem}

We partially extend this result for rings with zero-divisors. Our proof is just an adaptation of the original proof.

\begin{theorem} \label{17}
Let $A$ be a Noetherian ring. If  $B$ is an integral GD-overring of $A$, then  the non-minimal regular prime ideals of $A$  are unibranched in $B$.
\end{theorem}
\begin{proof}
Deny. There exist $Q_1\not\subseteq Q_2$  prime ideals of $B$ such
$Q_1\cap A=Q_2\cap A=P$ where $P$ is a non-minimal regular prime. Let $u\in Q_1-Q_2$. Since 'lying over' holds for $A[u]\subseteq B$, it is easy to see that  $A\subseteq A[u]$ satisfies going down. Replacing $B$ by $A[u]$, we may assume that $B=A[u]$. By going down, it follows that  $Q_1$ is a non-minimal regular prime ideal. As $A\subseteq B$ is a finite extension, the
 conductor  $I=A:B$  is a regular ideal of $A$ and $B$.
Note that $I\subseteq P$, otherwise $P$ would be unibranched in $B$.
Let $W=\{N_1$,...,$N_s\}$ be the (possibly empty) set of minimal regular  primes $H$ of $B$ such that  $I\subseteq H\subseteq Q_1$. Note that $W$ is finite because every $H\in W$ is minimal over $I$.
By prime avoidance, we can find a regular element $x\in Q_1- (Q_2\cup N_1\cup\cdots \cup N_s)$.
By Principal Ideal Theorem, there exists    a minimal regular  prime ideal  $N$ such that $xB\subseteq N\subseteq Q_1$.
By going down, there exists a prime $N'\subseteq Q_2$ of $B$ such that $N'\cap A=N\cap A:=M$.
As $x\in N$, we get $N\not\subseteq Q_2$, so $N\neq N'$, hence $M$ is not unibranched in $B$, thus $I\subseteq M\subseteq N$.
 It follows that $N\in W$, which is a  contradiction.
\end{proof}

We extend Theorem \ref{21} for rings with zero-divisors.  We say that an overring $B$ of a ring $A$ is a {\em unibranched overring} if the extension $A\subseteq B$ is unibranched.

\begin{theorem}\label{28}
Let $A$ be a Noetherian ring  with integral closure $A'$.
Assume that  $A_{(P\cap A)}$ is a DVR
for every minimal regular prime $P$ of $A'$.
The following are equivalent.

$(a)$ $A$ is perinormal.

$(b)$ Whenever $P$ is a prime ideal of $A$ and $B$ is an integral unibranched overring of $A_{(P)}$, we have  $A_{(P)}= B$.
\end{theorem}
\begin{proof}
$(a) \Rightarrow (b)$ follows from Propositions \ref{4} and \ref{5}.

$(b) \Rightarrow (a)$. Let $B$ be a GD-overring of $A$, $N$ be a prime ideal of $B$ and $M:=N\cap A$. By Lemma \ref{2}, it suffices to prove that $A_{(M)}=B_{(N)}$, because our rings are Marot.
By Lemma \ref{9}, the  map  ${Spec} (B_{(N)})\rightarrow Spec (A_{(M)})$ is surjective. Note that $A_{(M)}\subseteq B_{(N)}$ satisfies going down and our assumption on $A$ is inherited by $A_{(M)}$. Therefore, we may replace $A$ by $A_{(M)}$ and hence assume that the  map  ${Spec} (B)\rightarrow Spec (A)$ is surjective.
Let $P'$ be a minimal regular prime ideal of $A'$, set $P=P'\cap A$ and pick $Q\in Spec(B)$ such that $Q\cap A=P$. From the following three properties: $A_{(P)}\subseteq B_{(Q)}$, $A_{(P)}$ is a DVR and $QB_{(Q)}$ lies over $PA_{(P)}$, we get $B\subseteq B_{(Q)}=A_{(P)}={A'}_{(P')}$, cf. Theorem \ref{10}. Since $A'$ is a Krull ring, we get
$B\subseteq \cap_{P'}{A'}_{(P')}=A'$, so $B\subseteq A'$.
By Theorem \ref{17}, the non-minimal regular prime ideals of $A$ are unibranched in $B$. We claim that the minimal regular prime ideals of $A$ are unibranched in $B$. Indeed, let $P$ be a minimal regular prime ideal of $A$ and $Q\in Spec(A')$ lying over $P$. Since INC holds for $A\subseteq A'$, it follows that $Q$ is a minimal regular prime ideal of $A'$, hence $A_{(P)}$ is a DVR, by our hypothesis on $A$.
Now it is easy to see that if $Q\in Spec(B)$ lies over $P$, then $A_{(P)}=B_{(Q)}$, thus $P$ is unibranched in $B$. Finally, from Remark \ref{25}, the non-regular prime ideals of $A$ are unibranched in $B$.
Thus $A\subseteq B$ is unibranched and we get $A=B$ by our hypothesis.
\end{proof}

\noindent\textbf{Acknowledgement}\textit{.} The first author gratefully acknowledges the warm
hospitality of the Abdus Salam School of Mathematical Sciences GC University Lahore during his  visits in the period 2006-2015.
The second author is highly grateful to ASSMS GC University Lahore, Pakistan in supporting and facilitating this research.

%
%

\end{document}